\documentclass[12pt,english]{article}

\usepackage{amsmath,amsfonts,amssymb,amsthm,braket}

\usepackage{ifxetex,ifluatex}
\newif\ifuseunicodemath
\ifxetex \useunicodemathtrue \else \ifluatex \useunicodemathtrue \else \useunicodemathfalse \fi \fi
\ifuseunicodemath
\usepackage{polyglossia}
\setmainlanguage{english}

\usepackage[colon=literal]{unicode-math}
\else
\usepackage[utf8]{inputenc}
\usepackage[T1]{fontenc}
\usepackage[english]{babel}
\usepackage[mathscr]{euscript}
\fi

\usepackage[draft]{fixme}
\usepackage[unicode]{hyperref}
\usepackage{tikz}
\usetikzlibrary{matrix}

\newtheorem{theorem}{Theorem}
\newtheorem{corollary}[theorem]{Corollary}
\newtheorem{lemma}[theorem]{Lemma}
\theoremstyle{definition}
\newtheorem{definition}[theorem]{Definition}
\theoremstyle{remark}
\newtheorem{remark}[theorem]{Remark}

\newcommand{\Lfavg}{\overline{L}_f}
\newcommand{\Amax}{A_{\mathrm{max}}}
\newcommand{\AM}{A_{\mathrm{M}}}
\newcommand{\Lfmax}{\widehat L_f}

\DeclareMathOperator{\diam}{diam}
\DeclareMathOperator{\dist}{d}

\newcommand{\bbT}{\mathbb T}
\newcommand{\perturbed}{\mathcal G}
\newcommand{\rectified}{G}
\newcounter{step}

\usepackage{etoolbox}
\AtBeginEnvironment{proof}{\setcounter{step}{0}}
\newcommand{\step}[1]{\refstepcounter{step}\par\textbf{Step \thestep.} \emph{#1}.}

\FXRegisterAuthor{YK}{aYK}{Yu.K.}

\hypersetup{%
pdfkeywords={Partially hyperbolic dynamics, Anosov diffeomorphism, Skew product},%
pdfauthor={Yury G. Kudryashov},%
pdftitle={Bony attractors in higher dimension}%
}
\author{Yury G. Kudryashov\footnote{Research carried out within “The National Research University Higher School of Economics' Academic Fund Program” in 2012–2013, research grant No. 11-01-0236.}\footnote{Address: National Research University Higher School of Economics, Russia, Moscow, Myasnitskaya str., 20}}
\title{Bony attractors in higher dimension}

\begin{document}
\maketitle
\tableofcontents
\section{Introduction}
This article is devoted to study of possible geometric structures of attractors of typical dynamical systems.

In the most simple cases, an attractor of a dynamical system is a union of finite set of smooth manifolds. There are widely known examples of systems whose attractors are locally homeomorphic to a Cartesian product of a Cantor set by a manifold (say, Smale–Williams solenoid map) or a “Cantor book” (Lorenz attractor).

Recently, there emerged a few examples of locally typical dynamical systems having more complicated attractors. In particular, I \cite{bony-step} introduced the notion of a bony attractor, and Díaz with various co-authors \cite{DiazGelfert,DiazGelfertRams,DiazHorita} introduced a similar notion of a porcupine horseshoe.

In \cite{bony-step}, bony attractors were found in a rather artificial space of \emph{step skew products} over a Bernoulli shift. Later in PhD thesis \cite{YK-these} these results were extended to the set of diffeomorphisms of the three-dimensional torus. In this article, we extend the results to the set of diffeomorphisms of $\bbT^2\times S^d$ for any sphere $S^d$. We also simplify the proofs and fix some minor problems found in the proofs.

In the next Section, we shall introduce the required notions and notation, and formulate the main theorem. After that we shall prove the assertions of main theorem one by one.

\section{Required notions, notation and main theorem}
\subsection{Maximal attractor and the likely limit set}
First, recall two formalizations of the notion of attractor, the \emph{likely limit set} and the \emph{maximal attractor}.
\begin{definition}
	Let $F:X\to X$ be a continuous dynamical system. The \emph{$\omega $-limit set} $\omega (x)$ of a point $x\in X$ is the set of limit points of the positive semi-orbit $F^n(x)$.
\end{definition}
\begin{definition}
	[Milnor, \cite{Milnor2}]
	The \emph{likely limit set} of a dynamical system $F:X\to X$ on a topological space equipped with a measure is the minimal closed set $\AM(F)$ such that $F^n(x)\to \AM(F)$ as $n\to \infty $ for almost every point $x\in X$. Equivalently, $\AM(F)$ is the minimal closed set such that $\omega (x)\subset \AM(F)$ for almost every point $x\in X$.
\end{definition}
\begin{definition}
	Given a dynamical system $F:X\to X$, $F(X)\Subset X$, the \emph{maximal attractor} of $F$ is the intersection
	$$
		\Amax(F)=\bigcap _{n\geq 0} F^n(X).
	$$
\end{definition}
It is easy to show that $\AM(F)\subset \Amax(F)$. Indeed, $\Amax(F)$ includes the $\omega $-limit sets of \emph{all} points $x\in X$.

\subsection{Main Theorem}
Now we are ready to formulate the main theorem.
\begin{theorem}
	[Main Theorem]
	\label{thm:main}
	There exists a non-empty open subset $\mathscr U$ of the set of $C^2$ diffeomorphisms on $X=\bbT^2\times S^d$ such that for every diffeomorphism $\perturbed\in \mathscr U$ the following holds.
	\begin{enumerate}
		\item There exists a $\perturbed$-invariant topological fibration $\mathscr F^c$ of $X$ such that each fiber is diffeomorphic to $S^d$.
		\item For an uncountable set of fibers $\mathscr F^c(x)$ of $\mathscr F^c$, the likely limit set $\AM$ intersects $\mathscr F^c(x)$ on a set with non-empty interior. For such fibers, we shall say that the closure of the interior of $\AM\cap \mathscr F^c(x)$ is a \emph{bone}.
		\item The union of fibers that contain bones is dense in $X$.
		\item For an uncountable set of fibers $\mathscr F^c(x)$ of $\mathscr F^c$, the likely limit set $\AM$ intersects $\mathscr F^c(x)$ on a single point. We shall say that the union of these points is the \emph{graph part} of the attractor.
		\item The graph part of attractor is dense in $\AM$.
		\item $\mu \AM=0$; moreover, the Hausdorff dimension of $\AM$ is less than $d+2$.
		\item There exists a disk $D\subset S^d$ such that $\perturbed(\bbT^2\times D)\Subset \bbT^2\times D$ and the maximal attractor $\Amax=\bigcap _{i\geq 0} \perturbed^i(\bbT^2\times D)$ of the restriction of $\perturbed$ to $\bbT^2\times D$ coincides with $\AM$.
	\end{enumerate}
\end{theorem}
\subsection{(Partially) hyperbolic maps}
Next, we shall need some notions from partial hyperbolic theory. For a more detailed introduction, see \cite{Pesin}.
\begin{definition}
	\label{def:hyperbolic-broad}
	A smooth dynamical system $F:X\to X$ on a manifold is called \emph{partially hyperbolic in a broad sense} if there exist $\mu <\lambda $, $c>0$ and two invariant distributions $E^s_x\subset T_xX$ and $E^u_x\subset T_xX$, $dF_x(E^{s,u}_x)=E^{s,u}_{F(x)}$, such that $T_xX=E^s_x\oplus E^u_x$ and
	$$
		\| dF^n_x|_{E^s_x}\| \leq c\mu ^n,\quad \| dF^{-n}|_{E^u_x}\| \leq c\lambda ^{-n}.
	$$
\end{definition}
\begin{definition}
	A smooth dynamical system $F:X\to X$ is called \emph{partially hyperbolic in a strict sense} if there exists $C>0$,
	$$
		0<\lambda _1\leq \mu _1<\lambda _2\leq \mu _2<\lambda _3\leq \mu _3,\quad \mu _1<1, \lambda _3>1
	$$
	and an invariant decomposition
	$$
		T_xX=E^s_x\oplus E^c_x\oplus E^u_x, \quad dF(E^{s,c,u}_x)=E^{s,c,u}_{F(x)}
	$$
	such that for $n>0$ we have
	\begin{gather*}
		C^{-1}\lambda _1^n\| v\| \leq \| dF^nv\| \leq C\mu _1^n\| v\| ,\quad v\in E^s(x),\\
		C^{-1}\lambda _2^n\| v\| \leq \| dF^nv\| \leq C\mu _2^n\| v\| ,\quad v\in E^c(x),\\
		C^{-1}\lambda _3^n\| v\| \leq \| dF^nv\| \leq C\mu _3^n\| v\| ,\quad v\in E^u(x).
	\end{gather*}
\end{definition}
The distributions $E^s$ and $E^u$ are always integrable. The corresponding foliations are called \emph{stable} and \emph{unstable}. If $E^c$ is integrable as well, then the corresponding foliation is called \emph{central}.
\subsection{Ilyashenko–Gorodetski–Negut strategy}
Recall that a \emph{skew product} over a map $A:B\to B$ with fiber $M$ is a map $F:X\to X$, $X=B\times M$, of the form,
$$
	F(b, x)=(A(b), f_b(x)),
$$
i.e., a map preserving the vertical fibration $\set{b}\times M$.

Consider a skew product $F:X\to X$ over a linear hyperbolic diffeomorphism $A:\bbT^2\to \bbT^2$. Suppose that $F$ is partially hyperbolic in a strict sense with $\set{b}\times M$, $b\in B$, as central fibers. Moreover, suppose that the \emph{modified dominated splitting} condition holds,
\begin{equation*}
	\max\left(\lambda ^{-1}+\left\| \frac{\partial f_b(x)}{\partial b}\right\| _{C^0(X)}, \left\| \frac{\partial f_b(x)}{\partial x}\right\| _{C^0(X)}\right)<\lambda ,
\end{equation*}
where $\lambda $ is the greater eigenvalue of $A$.

The following theorem is a particular case of Ilyashenko–Negut Theorem \cite{YuI-Neg2} which is in turn based on the earlier research by A. Gorodetski and Ilyashenko \cite{gorodetski,Gor-YuI2,Gor-YuI-Robust}.
\begin{theorem}
	[Yu. Ilyashenko, A. Negut, \cite{YuI-Neg2}]
	Let $F$ be a $C^2$ skew product over $A:\bbT^2\to \bbT^2$, and $F$ satisfies the assumptions stated above. Then for $\rho $ small enough and a $C^2$ diffeomorphism $\perturbed:X\to X$ which is $\rho $-close to $F$ in $C^1(X)$ the following holds.
	\begin{itemize}
		\item There exists a continuous map $p:X\to \bbT^2$ such that $p\circ \perturbed=A\circ p$.
		\item The map $H:(b, x)\mapsto (p(b, x), x)$ is a homeomorphism that conjugates $\perturbed$ to a continuous skew product $\rectified$ over $A$.
		\item The fiber maps $g_b$ of $\rectified$ are smooth and are $C^1$ $O(\rho )$-close to those of $F$.
		\item The maps $H$, $H^{-1}$ and $\rectified$ are Hölder continuous in $b$ with exponent $1-O(\rho )$.
	\end{itemize}
\end{theorem}

This theorem allows us to construct open sets of diffeomorphisms having an interesting property. The strategy involves the following steps.
\begin{itemize}
	\item Construct a skew product $F$ over a linear Anosov diffeomorphism that has the properties we are interested in.
	\item Consider a small perturbation $\perturbed$ of $F$ in the space of diffeomorphisms.
	\item Use Gorodetski–Ilyashenko–Negut Theorem to obtain a skew product $\rectified$ conjugated to $\perturbed$.
	\item Prove that $\rectified$ has the properties we are interested in.
	\item Use the Hölder continuity of $p$ to show that $\perturbed$ has these properties as well.
\end{itemize}

This strategy introduced by Yu. Ilyashenko and A. Gorodetski in \cite{Gor-YuI-Robust,Gor-YuI2} and further developed in \cite{gorodetski}. The strategy was successfully used by various authors \cite{GIKN,Yui-Neg,YuI-Kl-Salt,Klep-Salt} to obtain open sets of diffeomorphisms having non-trivial properties.
\subsection{Notation}
We will frequently study the iterations of fiber maps. We will use the following notation throughout this paper.
\begin{align*}
	f_{b,n}(x)&=\pi _M(F^n(b, x))=f_{h^{n-1}(b)}\circ \dots \circ f_b(x);\\
	M_{b,n}&=f_{h^{-n}(b),n}(M);\\
	\Amax&=\bigcap _{n\geq 0}F^n(X);\\
	M_b&=\set{x|(b, x)\in \Amax}=\bigcap _{n\geq 0}M_{b,n}.
\end{align*}

\section{Construction of \texorpdfstring{$\mathscr U$}{U}}
Let $A$ be the linear Anosov diffeomorphism given by the matrix $\left(\begin{smallmatrix}m&m+1\\m-1&m\end{smallmatrix}\right)$, where $m$ is a large natural number that we will choose later. Consider the following Markov partition for this shift \cite{AdlerWeiss} (a description with picture is available in \cite{Klimenko,YK-these}).

First, let us split the torus into two parallelograms $Q_1$ and $Q_2$ with sides parallel to the eigenvectors of $A$ as shown in Figure \ref{fig:pre-markov}. This is a pre-markov partition. Then take the image of this partition under $A$, and draw both the original partition and its preimage under $A$ in the same picture. One can show that the intersections of the parallelograms of the initial pre-Markov partition with their images under $A$ form a Markov partition for $A$.

\begin{figure}[h]
	\centering\includegraphics{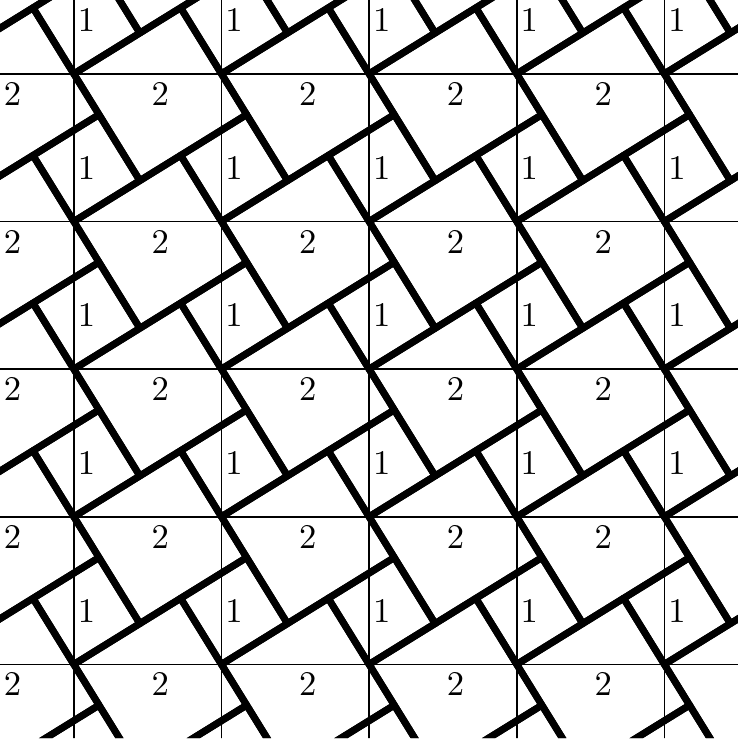}
	\caption{A pre-markov partition for $m=3$}\label{fig:pre-markov}
\end{figure}

Let $D$ be the unit disk in $\mathbb R ^d$. We will fix the fiber maps over some parallelograms of the Markov partition, then extend our map to other parallelograms.

Fix a regular simplex $p_0\dots p_d\subset D$ with side $0.5$ whose center of mass is located at the origin and a small positive number $\varepsilon $. Let $f_i$, $0\leq i\leq d$ be the linear contractions to $p_i$ with coefficient $1-\varepsilon $. Let $f_{d+1}$ be a map such that
\begin{itemize}
	\item the origin is a repellor for $f_{d+1}$, $f_{d+1}(0)=0$, $\| Df^{-1}_{d+1}(0)\| <1$;
	\item $f_{d+1}$ contracts outside a small neighborhood of the origin;
	\item $\max \| Df_{d+1}\| <1+\varepsilon $.
\end{itemize}
We also put $f_{d+2}=f_{d+1}$.

For $m$ large enough, we can choose $2d+6$ parallelograms of Markov partition $R_{ij}$, $i=0, \dots , d+2$, $j=1, 2$, such that
\begin{itemize}
	\item $R_{ij}$ is of type $(j, j)$;
	\item the distance between $R_{i, j}$ and $R_{i', j'}$ is at least $\frac{1}{10d}$ provided that $i\neq i'$.
\end{itemize}

Now we put $f_b(x)=f_i(x)$ for $b\in R_{ij}$, and extend it to a skew product on $\bbT^2\times D$ such that
\begin{itemize}
	\item all fiber maps $f_b$ are convex combinations of $f_i$;
	\item $F(b, x)$ is Lipschitz continuous in $b$ with constant $20d\varepsilon $;
	\item $f_b$ contracts in average,
	$$
		\int _{\bbT^2} \log\max_x \| Df_b(x)\| \,db<0.
	$$
\end{itemize}

Finally, in order to extend $F$ to a skew product $\widetilde F$ on $\bbT^2\times S^d$ we attach another disk $D'$ to $D$, and extend each fiber map $f_b$ by a uniformly expanding map on $D'$.

We will show that a small neighborhood of $\widetilde F$ in the space of $C^2$ diffeomorphisms of $\bbT^2\times S^d$ satisfies all the assertions of the Main Theorem.
\begin{theorem}
	\label{thm:bony-example}
	There exists $\rho $ such that the $\rho $-neighborhood of $\widetilde F$ constructed above in $C^1$ topology satisfies all assertions of the Main Theorem.
\end{theorem}
\begin{remark}
	The first assertion immediately follows from Ilyashenko–Negut Theorem stated above.

	Below we shall prove that for every $\widetilde\perturbed\in \mathscr U$, the \emph{maximal} attractors of $\perturbed=\widetilde\perturbed|_{\bbT^2\times D}$ satisfies all assertions of the Main Theorem, then prove that this maximal attractor coincides with the likely limit set of $\perturbed$ (and, hence, of $\widetilde\perturbed$).

	We prove most assertions of the Main Theorem for a wider class of diffeomorphisms. In this case, we formulate exact assumptions on $\perturbed$ (or $\rectified$) not using settings from Theorem \ref{thm:bony-example}.
\end{remark}
\section{Existence of bones}
The existence and density of the leaves containing bones will be based on the following lemma.
\begin{lemma}
	\label{lem:bones-exist}
	Let $G:X\to X$, $X=B\times M$, be a continuous skew product over a hyperbolic diffeomorphism $A:B\to B$. Let $b\in B$ be one of the periodic points of $A$, $A^q(b)=b$. Suppose that the iterated fiber map $g_{b,q}$ has a domain $U\subset M$ such that $U\Subset g_{b,q}(U)$. Then $M_b\supset U$, and $M_{b'}$ has non-empty interior for every point $b'\in B$ of the unstable manifold of $b$.
\end{lemma}
\begin{proof}
	The first assertion is trivial, $g_{b,qn}(M)\supset g_{b,qn}(U)=g_{b,qn-q}(g_{b,q}(U))\supset g_{b,q(n-1)}(U)\supset \dots \supset U$, thus $M_b\supset U$.

	Let us prove the second assertion. Since $U\Subset g_{b,q}(U)$, there exists a neighborhood $V\subset B$, $b\in V$, such that $U\Subset g_{\widetilde b,q}(U)$ for every $\widetilde b\in V$. Therefore, for a point $b'$ such that $h^{-qn}(b')\to b$ as $n\to \infty $, there exists $n_0$ such that $A^{-qn}(b')\in V$ for $n\geq n_0$. Due to the arguments from the first paragraph of the proof, $M_{A^{-qn_0}(b')}\supset U$, hence $M_{b'}=g_{A^{-qn_0}(b'),qn_0}(M_{A^{-qn_0}(b')})\supset g_{A^{-qn_0}(b'),qn_0}(U)$.
\end{proof}
\begin{corollary}
	For each map $\perturbed\in \mathscr U$, the \emph{maximal attractor} $\Amax$ includes an uncountable set of bones, and the set of leaves that include bones is dense in the phase space.
\end{corollary}
\begin{proof}
	Consider the rectified map $\rectified$. Since both rectangles $R_{d+1,1}$ and $R_{d+2,1}$ are of type $(1, 1)$, there is an uncountable set of periodic orbits that never leave the union $R_{d+1,1}\cup R_{d+2,1}$. Clearly, each periodic point of this type satisfies the assumptions of the previous lemma.

	Finally, there exists an uncountable set $\mathcal B$ of periodic points of $A$ such that for any point $b'$ in the unstable fiber of a point $b\in \mathcal B$ the maximal attractor intersects $\set{b'}\times D$ on a set with non-empty interior.
\end{proof}
\section{Hausdorff dimension of the attractor}
\begin{lemma}
	Let $A:\bbT^2\to \bbT^2$ be a linear Anosov diffeomorphism with eigenvalues $\lambda ^{\pm 1}$, $\lambda >1$. Consider a continuous skew product $G:X\to X$, $X=\bbT^2\times M$ over $A$, where $M$ is a $d$-dimensional manifold. Suppose each fiber map $g_b$ is Lipschitz continuous,
	$$
		\dist_M(g_b(m), g_b(m'))\leq L_b\dist_M(m, m'),
	$$
	and the fiber maps depend Hölder continuously on the point in the base,
	$$
		\dist_M(g_b(m), g_{b'}(m))\leq C_H\dist_{\bbT^2}(b, b')^\alpha .
	$$
	Finally, suppose that the fiber maps contract in average,
	$$
		\Lfavg :=  \exp\left(\int \log L_b\,d\mu \right) < 1.
	$$

	Then $\dim_H \Amax<\dim X$. Moreover, $\dim_H \Amax<\dim X-\varepsilon $, where $\varepsilon =\varepsilon (\alpha , b\mapsto L_b)$ is continuous in $\alpha $.
\end{lemma}
\begin{proof}
	[Plan of the proof]
	Let $n$ be a large natural number, $\delta $ be a small positive number. Consider a covering $\set{B_l}$ of the base by $\asymp \delta ^{-2}$ parallelogram of size $\asymp \delta $ with sides parallel to the eigenvectors of $A$. Since the fiber maps contract in average, $G^n(X)$ intersects most vertical fibers on sets of exponentially small diameter,
	$$
		\mu \Set{b\in \bbT^2|\diam M_{b,n} > L^n}<e^{-\beta n},\quad \text{where $\beta >0$, $\Lfavg<L<1$}.
	$$
	Therefore, most parallelograms $B_l$ contain a point $b\in B_l$ such that $\diam M_{b,n}<L^n$. On the other hand, if $\delta $ is small enough, the restriction of $G^n$ to $A^{-n}(B_l)\times M$ sends each horizontal plaque $A^{-n}(B_l)\times \set{m}$ to an “almost” horizontal plaque, thus the projection of $G^n(X)\cap (B_l\times M)$ to $M$ is included by a small neighborhood of any $M_{b,n}$, $b\in B_l$.
	
	Finally, for most parallelograms the set $M_{B_l,n}$ has an exponentially small diameter, and we can cover $B_l\times M_{B_l,n}$ by $\ll \delta ^{-d}$ balls of diameter $\delta $. For the rest of parallelograms, we just cover $B_l\times M$ by $\asymp \delta ^{-d}$ balls of diameter $\delta $.
\end{proof}
\begin{proof}
	\step{The image of a small horizontal plaque and the choice of $\delta $} Let $B$ be a small subset of $\bbT^2$, $\diam B=\delta $, let $n$ be a natural number. Consider a horizontal plaque $A^{-n}(B)\times \set{m}$ and its image under $G^n$. Let us estimate the size of this image in the vertical direction, i.e., the diameter of the projection of $G^n(A^{-n}(B)\times \set{m})$ to $M$. Take two points $b, b'\in A^{-n}(B)$. Note that $\dist(A^i(b), A^i(b'))\leq \delta \lambda ^n$ for all $i=0,\dots ,n-1$. Therefore,
	$$
		\dist_M(g_{b,i+1}(m), g_{b',i+1}(m))\leq \Lfmax\dist_M(g_{b,i}(m), g_{b',i}(m))+C_H(\delta \lambda ^n)^\alpha ,
	$$
	where $\Lfmax=\max_{b\in B} L_b$. Substituting each inequality to the next one, we have,
	$$
		\dist_M(g_{b,n}(m), g_{b',n}(m))\leq C_H(\delta \lambda ^n)^\alpha (1+\Lfmax+\dots +\Lfmax^{n-1})<\frac{C_H}{\Lfmax-1}\Lfmax^n\delta ^\alpha \lambda ^{n\alpha }.
	$$
	Next, fix a number $\nu <1$ and put $\delta := \left(\frac{\nu }{\Lfmax \lambda ^\alpha }\right)^{\frac n\alpha }$. Then due to the previous inequality,
	$$
		\dist_M(g_{b,n}(m), g_{b',n}(m))<\frac{C_H}{\Lfmax-1}\nu ^n.
	$$

	\step{Estimate on the diameter of $M_{B_l,n}$} Due to the previous inequality,
	$$
		\diam(M_{B_l,n})\leq \min_{b\in B_l} \diam(M_{b,n})+\frac{2C_H}{\Lfmax-1}\nu ^n.
	$$
	Fix $L\in (\Lfavg, 1)$. Due to Special Ergodic Theorem (a version of Large Deviation Theorem, see \cite{Saltykov}), there exists $\beta >0$ and $C>0$ such that
	$$
		\mu \Set{b\in \bbT^2|\frac 1n\sum _{i=1}^n\log L_{A^{-i}(b)}>\log L}<Ce^{-\beta n}
	$$
	for $n$ large enough. Therefore, for all but at most $Ce^{-\beta n}\delta ^{-2}$ parallelograms $B_l$ we have
	$$
		\min_{b\in B_l} \diam(M_{b,n})\leq \diam M\times L^n,
	$$
	thus
	$$
		\diam(M_{B_l,n})\leq \diam M\times L^n+\frac{2C_H}{\Lfmax-1}\nu ^n.
	$$

	\step{Covering of $G^n(X)$} Let us construct a covering of $G^n(X)$ by balls of diameter $\delta $. If $B_l$ satisfies the previous inequality, then we cover $B_l\times M_{B_l,n}$ by $\precsim \delta ^{-d}\max(L, \nu )^{nd}$ balls of size $\delta $. Otherwise, we cover $B_l\times M$ by $\asymp \delta ^{-d}$ balls of size $\delta $. Finally, the number of balls of diameter $\delta $ used in this covering is at most
	$$
		C_2(\delta ^{-2}\times \delta ^{-d}\max(L, \nu )^{nd}+e^{-\beta n}\delta ^{-2}\times \delta ^{-d})\precsim C_2\delta ^{-d-2}\max(L^d, \nu ^d, e^{-\beta })^n,
	$$
	therefore the Hausdorff dimension of the maximal attractor is at most
	\begin{multline*}
		\dim_H\Amax\leq \liminf_{n\to \infty }\frac{\log(C_2\delta ^{-d-2}\max(L^d, \nu ^d, e^{-\beta })^n)}{-\log \delta }\\
		=d+2+\frac{\alpha \max(d\log L, d\log \nu , -\beta )}{\log\Lfmax + \alpha \log \lambda  - \log \nu }<d+2=\dim X.
	\end{multline*}
	Clearly, this inequality provides a lower bound for $\dim X-\dim_H\Amax$ that depends only on $\alpha $ and $b\mapsto L_b$.
\end{proof}
\begin{corollary}
	For a diffeomorphism $\perturbed\in \mathscr U$, the Hausdorff dimension of the maximal attractor is less than $d+2$.
\end{corollary}
\begin{proof}
	Let us apply the previous lemma for the original example $F$. Let $\varepsilon _0$ be the estimate on $d+2-\dim_H\Amax$ provided by the lemma.

	Recall that the fiber maps of the rectified map $\rectified$ are $C^1$ close to those of the original map $F$. Therefore, for a small enough perturbation we have the same upper estimate on $L_b$. Hence, if the distance between $F$ and $\perturbed$ is small enough,
	$$
		\dim_H\Amax(\rectified) < \alpha (d+2),
	$$
	thus
	$$
		\dim_H\Amax(\perturbed) < \alpha ^{-1}\dim_H\Amax(\rectified) < d+2.
	$$
\end{proof}
\section{Density of the graph}
\begin{lemma}
	For $m$ large enough, for any $\perturbed\in \mathscr U$ the graph part of the attractor is dense in $\Amax$.
\end{lemma}
\begin{proof}
	It is easy to show that the strong stable direction of $F$ has slope at most
	$$
		k_0=\left\| \frac{\partial f_b(x)}{\partial b}\right\| \left(\lambda -\left\| \frac{\partial f_b(x)}{\partial x}\right\| \right)^{-1}<\frac{20d\varepsilon }{\lambda -1-\varepsilon }.
	$$
	Choose a strong stable cone field such that each line inside a stable cone has slope at most $2k_0$. If the perturbation is small enough, this cone field is invariant under $\perturbed^{-1}$, hence the strong stable leaves of $\perturbed$ have slope at most $2k_0$.

	Consider the regular simplex $J$ with vertices $p_0'=0.5p_i$. For $m$ large enough, $2k_0<0.01\varepsilon $, hence for every point $x\in J$ there exists $i\in \set{0,\dots ,d}$ such that $f_i(J)$ includes the $4k_0$-neighborhood of $x$. Next, if the perturbation is small enough, then the same holds for all maps $f_b$, $b\in R_{ij}$.

	Now, let us prove that the graph part of $\Amax(\perturbed)$ is dense in $\Amax(\perturbed)$. Consider a point $(b_0, m_0)\in \Amax(\perturbed)$ and its small neighborhood $\mathcal V$. Let $\gamma _u$ be a small arc of the unstable leaf of $A$ passing through $p(b_0, m_0)$, $m_0\in V_M\subset D$ be a small neighborhood of $m_0$. Clearly, if both $\gamma _u$ and $V_M$ are small enough, $\mathcal V$ includes $H^{-1}(\gamma _u\times V_M)$. Therefore, $\mathcal V$ contains the saturation $\mathcal V'$ of this set by small arcs of the strongly stable leaves of $\perturbed$. Without loss of generality, we may assume that $p(\mathcal V')$ is a parallelogram with sides parallel to the eigenvectors of $A$.

	Let $N$ be a number such that the side of $A^{-N}(p(\mathcal V'))$ going in the stable direction has length at least $2$. Consider two cases.

	\emph{Case 1}. The preimages of $\set{b_0}\times V_M$ under $G^n$, $n\geq N$, never intersect $\bbT^2\times J$. In this case each inverse fiber map $g_{A^{-n-1}(b_0)}^{-1}$ expands on $g_{A^{-n}(b_0),n}^{-1}(V_M)$. Therefore, at most one point of $\set{b_0}\times V_M$ belongs to the maximal attractor of $G$. On the other hand, $(b_0, m_0)\in \Amax(G)$ and the intersection $(\set{b_0}\times D)\cap \Amax(G)$ is a connected set. Therefore, this intersection is a single point, i.e., $(b_0, m_0)$ belongs to the graph part of the attractor.

	\emph{Case 2}. There exists $n\geq N$ such that $G^{-n}(\set{b_0}\times V_M)\cap \bbT^2\times J$ is not empty. Denote by $(b_0, m_1)$ one of the points of $(\set{b_0}\times V_M)\cap G^n(\bbT^2\times J)$. Consider the strongly stable leaf of $\perturbed$ passing through $H^{-1}(b_0, m_1)$. Let us intersect this leaf with $\mathcal V'$, and take the connected component $\gamma _{ss}$ passing through $H^{-1}(b_0, m_1)$. Let us prove that $\gamma _{ss}$ contains a point of the graph part of $\Amax$.

	Consider the preimage of $\gamma _{ss}$ under $\perturbed^n$. Note that $p\circ \perturbed^{-n}\circ \gamma _{ss}$ is a segment on the strongly stable leaf of $A$ passing through $b_0$ of length at least $2$. Cutting $\gamma _{ss}$ if needed, we can and will assume that $p\circ \perturbed^{-n}\circ \gamma _{ss}$ has length exactly $2$, hence the projection of $\perturbed^{-n}\circ \gamma _{ss}$ to the fiber $D$ has diameter at most $4k_0$.

	Since $g_{A^{-n}(b_0),n}^{-1}(m_1)\in J$, there exists $i_0\in \set{0,\dots ,d}$ such that $G_b(J)$ includes the $4k_0$ neighborhood of $g_{A^{-n}(b_0),n}^{-1}(m_1)$ for every $b\in R_{i_0j}$. In particular, $G_b(J)$ includes the projection of $\perturbed^{-n}\circ \gamma _{ss}$ to the fiber. Recall that the length of $p\circ \perturbed^{-n}\circ \gamma _{ss}$ is equal to $2$, hence this curve intersects both “unstable direction” sides of one of the pre-markov parallelograms. Therefore, this curve intersects both “unstable direction” sides of one of the rectangles $R_{i_0j}$. Finally, we obtain an arc $\gamma _{ss}^{(n+1)}\subset \gamma _{ss}$ such that $\perturbed^{-n-1}(\gamma _{ss}^{(n+1)})\subset \bbT^2\times J$. Due to the construction of the Markov partition and the choice of $R_{ij}$, the image of $\perturbed^{-n-1}(\gamma _{ss}^{(n+1)})$ under $p$ intersects both “unstable direction” sides of the same pre-Markov rectangle. Hence, we can apply the same construction to $\gamma _{ss}^{(n+1)}$ and $n+1$, etc.

	Finally, we obtain a sequence of arcs
	$$
		\gamma _{ss}\supset \gamma _{ss}^{(n+1)}\supset \gamma _{ss}^{(n+2)}\supset \dots 
	$$
	such that $(b, m)\in \gamma _{ss}^{(n+k)}$ implies that $\perturbed^{-n-k}(b, m)\in \bbT^2\times J$ and $g_{A^{-n-k}(p(b, m))}$ is a contracting map with coefficient at most $1-\varepsilon $. Let $(b, m)$ be the unique point that belongs to all these arcs. Then $(b, m)\in \mathcal V\cap \Amax(\perturbed)$ and all fiber maps $g_{A^{-n}(p(b, m))}$, $n>N$ contract. Therefore, $(b, m)$ belongs to the graph part of $\Amax(\perturbed)$.
\end{proof}
\section{Coincidence of attractors}
In previous sections, we proved that the maximal attractor $\Amax(\perturbed)$ has all the properties stated in the Main Theorem for the likely limit set $\AM$. In this section, we shall prove that $\Amax(\perturbed)=\AM(\perturbed)$ thus finishing the proof of the Main Theorem.

In order to prove the coincidence of attractors, we shall show that $\AM(\perturbed)$ cannot be disjoint with a fiber of the invariant fibration $\mathcal F^c$. We shall need the following lemma. It must be known for ages, but I failed to find a reference. A very similar result was proved (though not formulated as an isolated statement) in \cite[p.~ 215]{BonDiazViana}. I would like to thank V.~ Kleptsyn who pointed me to this book. The following proof essentially repeats the last paragraph of the proof of Proposition 11.1 in this book, providing much more details.
\begin{lemma}
	Consider a dynamical system $\perturbed:X\to X$ partially hyperbolic in the broad sense, see \ref{def:hyperbolic-broad}. Suppose that $\lambda >1$ and $\dim E_x^u=1$. Consider a closed set $V\subset X$ such that $\perturbed(V)\subset V$. Then
	\begin{itemize}
		\item either $\mu V=0$,
		\item or $V$ includes an arc of a stable leaf of $\perturbed$.
	\end{itemize}
\end{lemma}
\begin{proof}
	Suppose $\mu V>0$. Take a Lebesgue point $p_0$ of $V$. Near $p_0$, take a smooth $1$-dimensional foliation such that the tangent lines to the leaves belong to the stable cone field. Due to Fubini Theorem, the intersection of $V$ with one of the leaves $\gamma $ has positive $1$-dimensional Lebesgue measure.

	Without loss of generality, we can assume that $0$ is a Lebesgue point of $\gamma ^{-1}(V)$. Take a small positive number $\delta $ such that
	$$
		\mu (\gamma ^{-1}(V)\cap (-\delta , \delta ))>2\delta (1-\varepsilon ).
	$$
	Let $n(\delta )$ be the least natural number such that the image $\gamma _\delta $ of the curve $\gamma |_{(-\delta , \delta )}$ under $\perturbed^{n(\delta )}$ is longer than one. Let $\gamma _\delta ':(0, l(\delta ))\to X$ be the curve $\gamma _\delta $ parametrized by arc length. The Denjoy Distortion Lemma implies that the distortion of the map $\perturbed^{n(\delta )}$ on $\gamma ((-\delta , \delta ))$ is bounded, hence
	$$
		\mu (\gamma _\delta '^{-1}(V)) > (1-C\varepsilon )l(\delta ).
	$$
	Consider the family of curves $\gamma _\delta '$, $\delta \to 0$. Due to Arzelà–Ascoli Theorem, this family has a limit point in the space of $C^1$-smooth curves. Denote by $\gamma _0$ the limit curve parametrized by arc length. Clearly, $\gamma _0$ is an arc of a leaf of the unstable foliation of $\perturbed$. Since $V$ is closed, the inequality above implies that $\mu (\gamma _0^{-1}(V))=1$, hence $V$ includes $\gamma _0$. This completes the proof of the lemma.
\end{proof}
\begin{lemma}
	Let $\perturbed:X\to X$ be a smooth diffeomorphism partially hyperbolic in the strict sense. Suppose that there exist a map $p(X):X\to B$ and a transitive Anosov diffeomorphism $A:B\to B$ such that
	\begin{itemize}
		\item $p\circ \perturbed = A\circ p$ and preimages $p^{-1}(b)$, $b\in B$ are compact sets;
		\item the preimages $p^{-1}(b)$, $b\in B$ are leaves of the central foliation (in particular, the central distribution is integrable);
		\item the strongly unstable foliation has dimension one, and projects to the unstable foliation of $A$.
	\end{itemize}
	Then the likely limit set $\AM(\perturbed)$ intersects each fiber $p^{-1}(b)$, $b\in B$, by at least one point.
\end{lemma}
\begin{proof}
	Suppose that there exists $b\in B$ such that $\AM(\perturbed)$ is disjoint with $p^{-1}(b)$. Then $\AM(\perturbed)$ is disjoint with a small neighborhood $p^{-1}(U)$, $b\in U\subset B$ of this fiber.

	Choose an open set $U'\Subset U$, and consider the set $V$ of points $x\in X$ such that the positive semi-orbit of $x$ never visits $p^{-1}(U')$,
	\begin{equation*}
		V=\Set{x\in X|\forall n\geq 0\,{\mathcal G}^n(x)\notin p^{-1}(U')}=\bigcap _{n\geq 0}{\mathcal G}^{-n}(X\setminus p^{-1}(U')).
	\end{equation*}
	Clearly, $V$ is a closed set. The union $\tilde V$ of all preimages ${\mathcal G}^{-n}(V)$ is the set of points $x$ that visit $p^{-1}(U')$ at most finitely many times. Since $U'\Subset U$, this union includes the set of points $x\in X$ whose $\omega $-limit sets are disjoint with $p^{-1}(U)$.

	Since $\AM(\perturbed)\cap p^{-1}(U)=\varnothing$, the set $\tilde V$ has full Lebesgue measure, hence $\mu V>0$. Due to the previous lemma, $V$ includes an arc of a leaf of the strongly unstable foliation, thus $p(V)$ includes an arc $\gamma _u$ of a leaf of the unstable foliation of $A$. By definition of $V$, none of the curves $A^n(\gamma _u)$ intersect $U'$ which is impossible. This contradiction proves the lemma.
\end{proof}
Finally, let us prove that in the settings of Theorem \ref{thm:bony-example}, $\AM(\perturbed)=\Amax(\perturbed)$. Due to Hirsch–Pugh–Shub and Ilyashenko–Gorodetski Theorems, $\perturbed$ satisfies all assumptions of the previous lemma, hence $\AM(\perturbed)$ intersects each leaf $p^{-1}(b)$ by at least one point. Therefore, $\AM(\perturbed)$ includes the graph part of $\Amax(\perturbed)$. Since the graph part of $\Amax(\perturbed)$ is dense in $\Amax(\perturbed)$, $\AM(\perturbed)=\Amax(\perturbed)$.

Finally, we proved all assertions of Main Theorem.
\section{Acknowledgements}
The author is grateful to Yu. Ilyashenko for statement of the problem. I also grateful to É. Ghys and V. Kleptsyn for fruitful discussions, and to my wife N.~Goncharuk for her patients and for a great help with preparation of the text.
\bibliographystyle{plain}
\bibliography{these-en}
\end{document}